\documentclass{article}
\pdfoutput=1

\usepackage{amsmath,amsthm,amssymb,verbatim,amscd}
\usepackage{verbatim}
\usepackage{graphicx}
\usepackage{color}
\usepackage{float}
\usepackage[tmargin=1.5in,bmargin=1.5in,rmargin=1.5in,lmargin=1.5in]{geometry}

\theoremstyle{plain}

\newtheorem{theorem}{\textrm{\textbf{Theorem}}}[section]
\newtheorem{corollary}[theorem]{\textrm{\textbf{Corollary}}}
\newtheorem{proposition}[theorem]{\textrm{\textbf{Proposition}}}
\newtheorem{lemma}[theorem]{\textrm{\textbf{Lemma}}}

\theoremstyle{definition}

\newtheorem{definition}[theorem]{\textrm{\textbf{Definition}}}

\theoremstyle{remark}

\newtheorem*{acknowledgement}{Acknowledgment}

\numberwithin{equation}{section}

\newcommand{\Keywords}[1]{\par\noindent{\small{\em Keywords\/}: #1}}
\newcommand{\AMSclass}[1]{\par\noindent{\small{\em AMS subject classification\/}: #1}}

\def\mc#1{\multicolumn{1}{|c}{#1}}

\title{Retaining positive definiteness in thresholded matrices}
\author{Dominique Guillot\\ Stanford University \and Bala Rajaratnam\\ Stanford University}

\begin{document}
\maketitle
\begin{abstract}
Positive definite (p.d.) matrices arise naturally in many areas within mathematics and also feature extensively in scientific applications. In modern high-dimensional applications, a common approach to finding sparse positive definite matrices is to threshold their small off-diagonal elements. This thresholding, sometimes referred to as \emph{hard-thresholding}, sets small elements to zero. Thresholding has the attractive property that the resulting matrices are sparse, and are thus easier to interpret and work with. In many applications, it is often required, and thus implicitly assumed, that thresholded matrices retain positive definiteness. In this paper we formally investigate the algebraic properties of p.d. matrices which are thresholded. We demonstrate that for positive definiteness to be preserved, the pattern of elements to be set to zero has to necessarily correspond to a graph which is a union of disconnected complete components. This result rigorously demonstrates that, except in special cases,  positive definiteness can be easily lost. We then proceed to demonstrate that the class of diagonally dominant matrices is not maximal in terms of retaining positive definiteness when thresholded. Consequently, we derive characterizations of matrices which retain positive definiteness when thresholded with respect to important classes of graphs. In particular, we demonstrate that retaining positive definiteness upon thresholding is governed by complex algebraic conditions. 
\end{abstract}
\ \\
\AMSclass{Primary: 15B48, Secondary: 05C05, 05C38}

\section{Introduction}
Positive definite matrices arise naturally in various settings. Concrete examples are found in the fields of probability, statistics and machine learning. Here they are often used to represent covariance (or correlation) matrices. These covariance matrices encode multivariate relationships in a random vector and feature prominently in procedures like principal component analysis (PCA), linear discriminant analysis (LDA), etc.. Unlike Euclidean space, the open cone of positive definite matrices is more difficult to analyze because of the complex relationships between the elements of the matrices. In modern high-dimensional applications, a common approach to finding sparse positive definite or sparse covariance matrices is to hard-threshold small off-diagonal elements of given positive definite matrices. For example, a common approach to finding statistically significant relationships between genes in a gene-gene interaction analysis is to threshold the corresponding correlation matrix of gene-gene associations (see \cite{Li_Horvath}, \cite{Zhang_Horvath}). Thresholding has the property of setting correlations which are small in absolute value to zero. This procedure therefore eliminates spurious correlations which could have risen purely due to random noise. In the process, only statistically significant correlations are retained. This thresholding approach has the attractive property that it is easily computed, i.e., it is a highly scalable procedure, and yields a smaller number of large correlation coefficients. The resulting thresholded matrices are sparse and are thus easier to interpret and manipulate. Furthermore, finding a regularized or sparse estimate of the correlation matrix is often undertaken with an ulterior goal in mind. Indeed, these thresholded matrices are used as regularized estimates of covariance matrices in various applications. In particular, they are ingredients in several procedures in statistics and machine learning, such as principal component analysis, linear discriminant analysis, canonical correlation analysis, etc. (see \cite{bickel_levina}, \cite{flexible_cov}). For these procedures to be widely applicable, it is often assumed implicitly that the thresholded matrices retain positive definiteness. The theoretical investigation of whether this is true is critical for the validity of such applications. Statistical properties of thresholded correlation matrices have been studied in the literature (see \cite{bickel_levina}), however, algebraic properties of such matrices have not been previously investigated. 

In this paper, we provide a rigorous answer to the aforementioned questions. In particular, we prove that, if positive definiteness is to be retained after thresholding, the pattern of elements to be thresholded has to necessarily correspond to a graph which is a union of disconnected complete components. We also consider the problem of thresholding positive definite matrices at a given level $\epsilon$, i.e., only to zero out elements which are less than $\epsilon$ in magnitude. In this case, we prove that only matrices which already have zeros according to a tree, will retain positive definiteness for all values of the thresholding parameter $\epsilon$. These results imply that positive definiteness can be easily lost due to hard-thresholding. Furthermore, it is well known that the class of diagonally dominant matrices with strictly positive diagonal entries is positive definite. Procedures such as hard-thresholding maintain diagonal dominance. Hence, we also investigate if positive definiteness is only retained by the class of diagonally dominant matrices, and proceed to demonstrate that this class is not maximal. We therefore proceed to identify algebraic relationships which characterize retaining positive definiteness for important classes of graphs. Our results formally demonstrate that thresholding approaches used in the literature can lead to sparse matrices which are no longer positive definite.

As a concrete example, let us threshold the $(2,3)$ and $(3,2)$ elements of the following symmetric matrix 
\[
A = \left(\begin{array}{ccc} 
4 & 3 & -3 \\
3 & 4 & -1 \\
-3 & -1 & 4
\end{array}\right) 
\quad 
A' = \left(\begin{array}{ccc} 
4 & 3 & -3 \\
3 & 4 & 0 \\
-3 & 0 & 4
\end{array}\right). 
\]
One can easily check that $A$ is positive definite, but $A'$ is not. So a fundamental question that arises is, when do we know for certain that $A' > 0$. If $A$ comes from certain special subclasses of the cone of p.d. matrices, the answer is immediate. For instance, if $A \in \mathcal{B}$ where $\mathcal{B} \subset \mathbb{P}^+$ is the class of diagonally dominant matrices, then $A'$ is always positive definite. As pointed out in the example above, this property is not true in general for an arbitrary positive definite matrix $A$. 

More formally, the elements to threshold are naturally encoded in a graph $G = (V,E)$ with $V = \{1,\dots,p\}$ and with the convention that $A_{ij}$ is set to zero if and only if $(i,j) \not\in E$. We denote the thresholded matrix by $A_G$. So when will $A_G > 0$ for certain, i.e., what classes of graphs will ensure positive definiteness. It is not immediately clear what $G$ should be. In this paper we give results of potentially great consequence, that only for a narrow class of graphs can we ensure positive definiteness. The statement of the main theorem in this paper is given below. 

\begin{theorem}
Let $A$ be an arbitrary symmetric matrix such that $A > 0$, i.e., $A \in \mathbb{P}^+$. Threshold $A$ according to a graph $G = (V,E)$ with the resulting thresholded matrix denoted by $A_G$. Then
\[
A_G > 0 \textrm{ for any } A \in \mathbb{P}^+ \Leftrightarrow G = \bigcup_{i=1}^\tau G_i \qquad \textrm{ for some } \tau \in \mathbb{N},  
\]
where $G_i$, $i=1, \dots, \tau$, denote disconnected and complete components of $G$. 
\end{theorem}
The $(\Leftarrow)$ part of the theorem is intuitive and straightforward but the $(\Rightarrow)$ part does come as somewhat of a surprise. It is also a stark reminder that indiscriminate or arbitrary thresholding of a positive definite matrix can quickly lead to loss of positive definiteness. 

The remainder of the paper is structured as follows. Section \ref{Section_Prelim} introduces basic terminology, notation and preliminaries. 
Section \ref{Section_Proof} provides a proof of the main theorem and further extensions. In particular, we consider the effect of thresholding when a positive definite matrix already has zeros according to a graph $G$. We also investigate the impact of thresholding when only elements smaller in magnitude than a prescribed value are thresholded, as is often done in practice. We also study the maximal class of matrices retaining positive definiteness and show that the set of diagonally dominant matrices is never maximal.  In Section \ref{Section_Decomp}, we characterize the class of matrices that retain positive definiteness when thresholded with respect to important classes of graphs. We discover that, even in simple cases, characterizing the maximal class leads to complicated algebraic relations. There is thus little hope of obtaining a characterization in terms of a simple property like diagonal dominance. 
\section{Preliminaries}\label{Section_Prelim}
\subsection{Graph theory}
Let $G=(V,E)$ be an undirected graph with vertex set $V = \{1, \dots, p\}$ and edge set $E$. Two vertices $a, b \in V$, $a \not= b$, are said to be \emph{adjacent} in $G$ if $(a,b) \in E$. We say that the graph $G'=(V', E')$ is a subgraph of  $G=(V, E)$,  denoted by  $G'\subset G$,  if $V'\subseteq V$ and  $E'\subset E$. In addition, if  $G'\subset G$  and  $E'=V'\times V'\cap E$, we say that $G'$ is an {induced}  subgraph of  $G$. We shall consider only induced subgraphs in what follows. For a subset  $A\subset V$, the induced subgraph  $G_A=(A, A\times A\cap E)$ is said to be the graph {induced} by $A$. A graph  $G$  is called \emph{complete} if every pair of vertices are adjacent. A  \emph{clique}  of  $G$  is an induced complete subgraph of $G$ that is not a proper subset of any other induced complete subgraphs  of  $G$. A \emph{path} of length  $k\geq 1$  from vertex $i$  to  $j$  is a finite sequence of distinct vertices  $v_0=i,\ldots, v_k=j$  in $V$  and edges $(v_0,v_1), \ldots, (v_{k-1}, v_k)\in E$. An \emph{$n$-{cycle}} in  $G$  is a path of length  $n$ with an additional edge connecting the two end points. A graph $G$ is called \emph{connected} if for any pair of distinct vertices $i, j\in V$ there exists a path between them.

Now let $A, B, C \subset V$ be three nonempty subsets of $V$. We say that $C$ \emph{separates} $A$ from $B$ if every path from a vertex $a \in A$ to a vertex $b \in B$ contains a vertex in $C$. The graph $G$ is said to be \emph{chordal} or \emph{decomposable} (\cite{Chavatal}, \cite{Golumbic}) if it does not contain a cycle of length $\geq 4$ as an induced subgraph. Alternatively, a graph $G$ is said to be \emph{chordal} or \emph{decomposable} if either $G$ is complete or if there exist subsets $A, B, C \subset V$ such that 
\begin{enumerate}
\item $V = A \cup B \cup C$; 
\item $C$ separates $A$ from $B$; 
\item $C$ is complete; 
\item $G_{A \cup C}$ and $G_{B \cup C}$ are decomposable.
\end{enumerate} 
A triple $(A,B,C)$ satisfying the first three properties above is said to be a \emph{decomposition} of $G$.  

Furthermore, let $C_1, C_2, \dots, C_k$ be an ordering of the cliques of a graph $G$. Define the \emph{history} up to clique $q$ as $H_q = C_1 \cup \dots \cup C_q$ and the \emph{separator} $S_q = C_q \cap H_{q-1}$. The ordering $\{C_1, \dots, C_k\}$ is said to be a \emph{perfect ordering} if for every $q=2, \dots, k$, there exists $p \leq q-1$ such that $S_q \subset C_p$. A well known result in graph theory is that every decomposable graph admits a perfect ordering of its cliques. Conversely, this property characterizes decomposable graphs. A special class of decomposable graphs are \emph{trees}. These are connected graphs on $n$ vertices with exactly $n-1$ edges. A tree can also be defined as a connected graph with no cycle of length $n \geq 3$, or a connected graph with a unique path between any two vertices. Alternatively, a tree is a connected decomposable graph with maximal clique size $2$.

In the remainder of the paper, we will assume that the vertices $V$ of $G$ are labeled by the set $\{1, 2, \dots, n\}$. As mentioned earlier, the graph induces a \emph{hard-thresholding} operation, mapping every symmetric matrix $A = (a_{ij}) \in M_n$ to a matrix $A_G$ in $M_n$ defined by 
\[
(A_G)_{ij} = \left\{\begin{array}{ll} a_{ij} & \textrm{if } (i,j) \in E \textrm{ or } i=j \\ 0 & \textrm{otherwise}\end{array}\right..
\]
We say that the matrix $A_G$ is obtained from $A$ by \emph{thresholding $A$ with respect to the graph $G$}.

\subsection{Linear Algebra}\label{Subsec_Linalg}
We denote by $\mathbb{P}^+$ the cone of positive matrices. Let $G = (V,E)$ be a graph and let
\[
\mathbb{P}_G = \{A \in \mathbb{P}^+ : a_{ij} = 0 \textrm{ if } (i,j) \not\in E, i \not=j\}
\]
be the set of positive definite matrices with fixed zeros according to $G$. 

A class of matrices which are guaranteed to retain positive definiteness upon thresholding by any arbitrary graph is the class of diagonally dominant matrices. We formalize this concept in the following well known results from linear algebra \cite{HJ_Matrix}. 

\begin{definition}[Strictly Diagonally Dominant matrices]
A matrix $A$ is said to be \emph{ strictly diagonally dominant} if for every $i=1,\dots,n$, 
\[
|a_{ii}| > \sum_{j \not = i} |a_{ij}|. 
\]
\end{definition}

\begin{theorem}[Gershgorin circle theorem, \cite{HJ_Matrix} - Theorem 6.1.1]
Let $A = (a_{ij}) \in M_p(\mathbb{R})$. For $i=1, \dots, p$, let $R_i = \sum_{j \not= i} |a_{ij}|$ and let $D_i = D(a_{ii},R_i)$ be the closed disc with center $a_{ii}$ and radius $R_i$. Then every eigenvalue of $A$ belongs to at least one of the discs $D_i$.  
\end{theorem}

\begin{corollary}[Positive definiteness of diagonally dominant matrices]
A strictly diagonally dominant matrix with positive diagonal entries is necessarily positive definite.
\end{corollary}


\begin{corollary}[Positive definiteness of thresholded strictly diagonally dominant matrices]
Let $G$ be an arbitrary graph and consider a strictly diagonally dominant matrix A. Then $A_G$ is positive definite.
\end{corollary}
The proof follows immediately from the Gershgorin circle theorem.\ \\

A natural question is whether positive definiteness is in general retained when matrices are thresholded. Moreover, are diagonally dominant matrices the maximal class which retains positive definiteness when thresholded ? We investigate these and related questions in this paper. 

Finally, recall that a $(p_1+p_2) \times (p_1 + p_2)$ symmetric block matrix 
\[
M = \left(\begin{array}{cc}A & B \\ B^t & D\end{array}\right)
\]
where $A \in M_{p_1 \times p_1}, B \in M_{p_1 \times p_2}$, and $D \in M_{p_2 \times p_2}$, is positive definite if and only if $D$ is positive definite and $S_1 = A - BD^{-1}B^t$ is positive definite. The matrix $S_1$ is called the \emph{Schur complement} of $D$ in $M$. Alternatively, $M$ is positive definite if and only if $A$ is positive definite and $S_2 = D-B^tA^{-1}B$ is positive definite. The matrix $S_2$ is called the \emph{Schur complement} of $A$ in $M$.


\section{Thresholding of arbitrary positive definite matrices}\label{Section_Proof}
We now proceed to prove the main result of this paper. 

\begin{theorem}\label{Th_Principal}
Let $A$ be an arbitrary symmetric matrix such that $A > 0$, i.e., $A \in \mathbb{P}^+$. Threshold $A$ with respect to the graph $G = (V,E)$ with the resulting thresholded matrix denoted by $A_G$. Then
\[
A_G > 0 \textrm{ for any } A \in \mathbb{P}^+ \Leftrightarrow G = \bigcup_{i=1}^\tau G_i \qquad \textrm{ for some } \tau \in \mathbb{N},  
\]
where $G_i$, $i=1, \dots, \tau$, are disconnected and complete components of $G$ (or equivalently $G$ is comprised of only complete disconnected components). 
\end{theorem}
\begin{proof}
($\Leftarrow)$ This direction is straightforward. Note that every principal submatrix of a positive definite matrix is also positive definite. So $A_G$ can be rearranged in a block diagonal form with each block $A_i > 0$ corresponding to each subgraph $G_i$. It follows easily that $A_G > 0$.
\ \\ 
($\Rightarrow$) We shall prove the contrapositive form, i.e., for a graph $G$ which is not a union of complete disconnected components, there exists a matrix $A \in \mathbb{P}^+$ such that $A_G \not> 0$. We shall use mathematical induction on $|V| = k$. Suppose first that $k=3$. We only need to check that the following graph, often termed as the $A_3$ graph, can lead to a loss of positive definiteness: 
\begin{center}
\includegraphics[width=3cm]{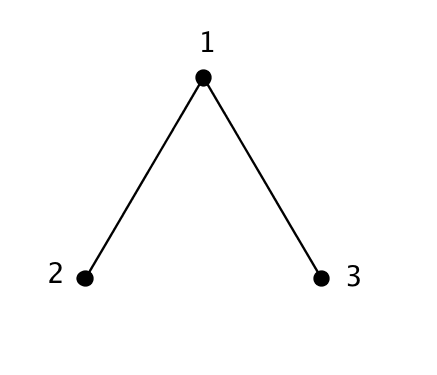}
\end{center}
Consider the following example: 
\[
A = \left(\begin{array}{ccc} 
4 & 3 & -3 \\
3 & 4 & -1 \\
-3 & -1 & 4
\end{array}\right) 
\quad 
A_G = \left(\begin{array}{ccc} 
4 & 3 & -3 \\
3 & 4 & 0 \\
-3 & 0 & 4
\end{array}\right). 
\]
One can easily check that $A > 0$, but $A_G$ is not. The theorem is hence true for $k=3$. 

Let us now assume the inductive hypothesis when $k=p-1$, i.e., let $G = (V, E)$ be such that $|V| = p-1$. Then 
\[
A_G > 0 \textrm{ for any } A_{(p-1) \times (p-1)} \in \mathbb{P}_{p-1}^+ \Rightarrow G = \bigcup_{i=1}^\tau G_i
\]
where $G_i$ are complete disconnected components. Let us now show that the result holds true when $k=p$. Let $G' = (V',E')$ be an arbitrary graph such that $|V'|=p$. Let $G = (V,E)$ be the subgraph of $G'$ induced by the vertices $\{1,\dots,p-1\}$. If a $p \times p$ positive definite matrix $A'$ retains positive definiteness when thresholded with respect to $G'$, then the $(p-1) \times (p-1)$ principal submatrix of $A'$ will also be positive definite when thresholded with respect to $G$. Therefore, if every matrix $A' \in \mathbb{P}^+$ stays positive definite when thresholded with respect to $G'$ then, by the inductive hypothesis, the graph $G$ can be decomposed as a union of complete disconnected components. So now on, let us assume that $G$ is a union of complete disconnected components and assume to the contrary that $G'$ is not a union of complete disconnected components. Since $G'$ is not a union of complete disconnected components, vertex $p$ can either a) be connected to exactly one complete connected component $C_1$ of $G$, where $C_1 \cup \{p\}$ is not a complete subgraph of $G'$ or b) to two or more connected components of $G$. In the first case, there exists vertices $x, y \in C_1$ such that $\{p,x,y\}$ induces the following subgraph from $G'$:
\begin{center}
\includegraphics[width=2.5cm]{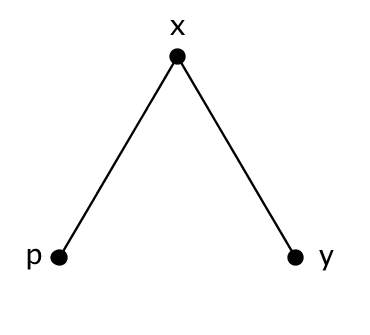}
\end{center}
In the second case, there exists vertices $x \in C_1$ and $y \in C_2$, where $C_1$ and $C_2$ are different connected components of $G$, such that $\{p,x,y\}$ induces the following subgraph from $G'$
\begin{center}
\includegraphics[width=2.5cm]{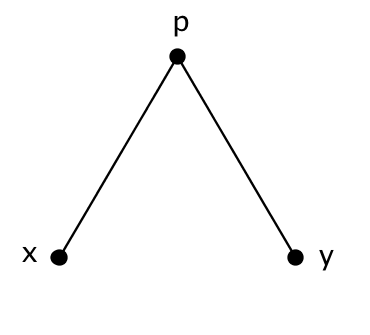}
\end{center}
 Let us construct $A' > 0$ as follows
\[
\bordermatrix{
 & p & x & y & \cr
p & a_{pp} & a_{px} & a_{py} & \dots \cr
x & a_{xp} & a_{xx} & a_{xy} & \dots \cr
y & a_{yp} & a_{yx} & a_{yy} & \dots \cr
 & \vdots & \vdots & \vdots & \ddots
}, 
\]
where the $3$ by $3$ upper left block corresponds to the example above for the $k=3$ case. It is possible to choose the other entries of $A'$ such that $A' > 0$ (a concrete example can be constructed by augmenting the $3$ by $3$ submatrix with a block diagonal identity matrix of dimension $p-3$). By the $k=3$ case, when $A'$ is thresholded with respect to $G'$, the $3$ by $3$ block loses positive definiteness, and consequently $A'_{G'} \not> 0$. This yields a contradiction to the initial hypothesis that $A'_{G'} > 0$. Therefore, the graph $G'$ decomposes as a union of complete disconnected components. 
\end{proof}
{\it Remarks: }\begin{enumerate}\item An alternative proof, which circumvents the inductive argument, would follow from first proving that the $A_3$ graph is necessarily an induced subgraph of $G' $ when $G'$ is not a union of complete disconnected components. The remainder of the proof would follow the same line of argument as above. 
\item Note that the proof for any arbitrary fixed dimension hinges on losing positive definiteness in the $p=3$ case. 
\item The result above holds true if we replace the class of positive definite matrices by the class of positive semidefinite matrices since the former set is a subset of the latter.  Also, the result remains valid if we restrict the set of positive semidefinite matrices to the subclass which are singular (i.e. matrices which are positive semidefinite but are not positive definite).
\end{enumerate}

 We formalize the extension of our results to the class of singular positive semidefinite matrices in the following corollary.  

\begin{corollary}
Let $A$ be an arbitrary symmetric matrix such that $A \geq 0$ and $\det A = 0$. Threshold $A$ according to a graph $G = (V,E)$ with the resulting thresholded matrix denoted by $A_G$. Then
\[
A_G \geq  0 \textrm{ for any } A \geq 0 \textrm{ with } \det A =0 \Leftrightarrow G = \bigcup_{i=1}^\tau G_i \qquad \textrm{ for some } \tau \in \mathbb{N},  
\]
where $G_i$, $i = 1, \dots, \tau$, are disconnected and complete components of $G$ (or equivalently $G$ is comprised of only complete disconnected components). 
\end{corollary}
\begin{proof}
($\Leftarrow$) This direction is clear since a principal submatrix of a positive semidefinite matrix is positive semidefinite. \ \\
($\Rightarrow$) We shall prove the contrapositive form. Suppose $G$ is not a union of complete disconnected components. According to Theorem \ref{Th_Principal}, there is a matrix $A > 0$ such that $A_G$ is not positive definite. Let $\lambda_1$ and $\eta_1$ be the smallest eigenvalues of $A$ and $A_G$ respectively. Note that $\lambda_1 > 0$ but $\eta_1 \leq 0$. Consider the matrix $B = A - \lambda_1 I$. This matrix satisfies $B \geq 0$ and $\det B = 0$. In addition, $B_G = A_G - \lambda_1 I$ and so the smallest eigenvalue of $B_G$ is $\eta_1 - \lambda_1 < 0$. Hence there exist a matrix $B \geq 0$ with $\det B = 0$ such that $B_G \not\geq 0$.  
\end{proof}

Theorem \ref{Th_Principal} demonstrates that ensuring positive definiteness after thresholding for the whole cone $\mathbb{P}^+$ is rather difficult; in the sense that only when $G$ is a union of complete disconnected components can we ensure positive definiteness of the resulting thresholded matrix. It is therefore natural to try to threshold a smaller class of matrices than the whole space of positive definite matrices to see if positive definiteness is retained. Consider for example the case where a positive definite matrix $A$ already has zeros according to a graph $G$. This will be the case, for example, when $A$ is a sparse positive definite matrix. If $A$ where to be thresholded, will it remain positive definite ? Specifically, for which subgraphs $H$ of $G$ will $A_H$ retain positive definiteness ? The next theorem significantly generalizes Theorem \ref{Th_Principal} and demonstrates that the answer is essentially the same. 

\begin{theorem}\label{Th_Hard_Thres_General}
Let $G = (V,E)$ be an undirected graph and let $H = (V,E')$ be a subgraph of $G$ i.e., $E' \subset E$. Then $A_H > 0$ for every $A \in \mathbb{P}_G$ if and only if $H = G_1 \cup \dots \cup G_k$ where $G_1, \dots, G_k$ are disconnected induced subgraphs of $G$. 
\end{theorem}
\begin{proof}
The $(\Leftarrow)$ part of the theorem is clear since a principal submatrix of a positive definite matrix is positive. \ \\
$(\Rightarrow)$ We shall prove the contrapositive form. Suppose $H$ cannot be written as a union of disconnected induced subgraphs of $G$. Let $H_1, \dots, H_m$ be the connected components of $H$. Since $H = H_1 \cup \dots \cup H_m$ and the components are disconnected, one of the $H_i$ is not an induced subgraph of $G$. Without loss of generality, assume $H_1$ is not an induced subgraph. Then there are two vertices $u, v \in H_1$ such that $(u,v) \not\in E'$, but $(u,v) \in E$. Since $H_1$ is connected, there is a path $u=v_1, \dots, v_k = v$ connecting $u$ and $v$ in $H_1$. Let $C$ be the cycle obtained by adding the edge $(u,v)$ to the end of this path. This cycle belongs to $G$. Therefore, when the matrix $A$ is thresholded with respect to $H$, a cycle belonging to $G$ is broken. In the remainder of the proof, we will prove that, for any cycle $C_n$ of length $n \geq 3$, there is a positive definite matrix $A = A_{C_n}$ such that $A$ loses positive definiteness when thresholded by removing an edge of $C_n$. This together with the fact that any principal submatrix of a positive definite matrix is also necessarily positive definite will prove the theorem. 

Let $C_n$ be a cycle of length $n \geq 3$. Consider the matrix 
\[
A_n = \left(\begin{array}{cccccccc}
\alpha & 1 & 0 & 0 & \dots & 0 & 0 & \mc{a}\\
1 & 2 & 1 & 0 & \dots & 0 & 0 & \mc{0}\\
0 & 1 & 2 & 1 & \dots & 0 & 0 & \mc{0}\\
\vdots & \vdots & \vdots & \vdots & \ddots & 2 & 1 & \mc{0}  \\
0 & 0 & 0 & 0 & 0 & 1 & 2 & \mc{b} \\
\cline{1-8}
a & 0 & 0 & 0 & \dots & 0 & b & \mc{\beta}
\end{array}\right)
\]
where the upper left block is the matrix of a path of length $n-1$ with off-diagonal elements equal to $1$ and diagonal elements equal to $2$ except the $(1,1)$ element which is equal to $\alpha$.  Clearly, $A_n = (A_n)_{C_n}$. We first claim that 
\begin{equation}\label{eq_det_an}
\det A_n = -(n-2) \beta + (n-1)\alpha \beta + (-1)^{n+1} 2 ab - (n-1)a^2 + (n-3)b^2 - (n-2) \alpha b^2. 
\end{equation}
We will prove this formula by induction on $n$. If $n=3$, the matrix reduces to 
\[
A_3 = \left(\begin{array}{ccc}
\alpha & 1 & a \\
1 & 2 & b \\
a & b & \beta
\end{array}\right)
\]
and
\begin{eqnarray*}	
\det A_3 &=& \alpha(2\beta - b^2) - (\beta - ab) + a(b-2a) \\
	      &=& -\beta + 2 \alpha \beta + 2ab - 2a^2 - \alpha b^2.
\end{eqnarray*}
So formula (\ref{eq_det_an}) is valid for $n=3$. Suppose, by the inductive hypothesis, that formula (\ref{eq_det_an}) is valid for $k=n-1$, and now consider the matrix $A_n$. The determinant of $A_n$ is given by $\beta$ times the determinant of the Schur complement $S$ of the lower right block. Notice that the matrix $S$ is equal to $A_{n-1}$ with $\alpha$ replaced by $\alpha - a^2/\beta$, $\beta$ replaced by $2-b^2/\beta$, $b$ by $1$, and $a$ by $-ab/\beta$.  Therefore, by the induction hypothesis, 
\begin{eqnarray*}
\det S &=& -(n-3)\left(2-\frac{b^2}{\beta}\right) + (n-2)\left(\alpha-\frac{a^2}{\beta}\right)\left(2-\frac{b^2}{\beta}\right) + (-1)^{n+1} 2\frac{ab}{\beta} \\ &-&(n-2)\frac{a^2b^2}{\beta^2} + (n-4) - (n-3)\left(\alpha-\frac{a^2}{\beta}\right) \\
&=& 2-n+\alpha(n-1) + \frac{b^2}{\beta} (n-3 - \alpha(n-2)) - \frac{a^2}{\beta} (n-1) + 2(-1)^{n+1} \frac{ab}{\beta}. 
\end{eqnarray*}
Hence,
\[
\det A_n = \beta \det S = -(n-2)\beta +(n-1)\alpha \beta + (-1)^{n+1} 2ab - (n-1)a^2 + (n-3)b^2 - (n-2)\alpha b^2, 
\]
and so, formula (\ref{eq_det_an}) is valid for $k=n$. The expression in (\ref{eq_det_an}) is therefore valid for every $n \geq 3$.  

Now, let $M_n = M_n(a,b)$ be the matrix $A_n$ with $\alpha = \beta =2$. We will prove that we can choose $a$ and $b$ such that $M_n$ is positive definite, but $M_n$ is not positive definite if $a$ is set to $0$. 

By using the formula for $\det A_n$, it is easily shown that the determinant of the $k$-th principal submatrix, corresponding to the upper left corner of $M_n$, is equal to
\[
 -2(k-2)  + 4(k-1)  + (k-3) - 2(k-2) = 1 + k > 0 \quad \textrm{ for } k=1, \dots, n-1. 
\] 
Therefore, the upper left $(n-1) \times (n-1)$ block of $M_n$ is always positive definite. The determinant of $M_n$ is given by
\[
p(a,b) = 2n -(n-1)a^2 - (n-1)b^2 + (-1)^{n+1}2 ab 
\]
Note that $p(a,b)$ is a concave quadratic equation in $a$ for fixed $b$. If the $(1,n)$ and $(n,1)$ elements are thresholded, i.e., if $a$ is set to $0$, the new determinant is given by 
\[
p(0,b):= q(b) = 2n - (n-1)b^2. 
\]
So in order for the thresholded matrix to lose positive definiteness, we must choose $b$ such that $q(b) < 0$ i.e., $b^2 > 2n/(n-1)$. Hence let
\[
b^2 = \frac{2n}{n-1} + \epsilon
\]
where $\epsilon > 0$. We claim that, if $\epsilon$ is small enough, we can choose $a$ such that $p(a,b) > 0$. Indeed, the discriminant of $p$ (seen as a function of $a$), when the above value of $b^2$ is substituted, is given by
\begin{eqnarray*}
\Delta &=& 4b^2 + 4(n-1)\left[2n-(n-1)b^2\right] \\
&=& 4\left(\frac{2n}{n-1}+\epsilon\right) - 4(n-1)^2\epsilon. 
\end{eqnarray*}
Hence we can choose $\epsilon > 0$ small enough such that $\Delta > 0$. With this choice of $\epsilon$, the polynomial $p(a,b)$ has two real roots and so we can choose $a$ such that $p(a,b) > 0$. Therefore, for these choices of $a$ and $b$, the matrix $M_n$ is positive definite, but loses positive definiteness if $a$ is thresholded to zero. The above arguments provide the needed example for the cycle $C_n$ and therefore proves the theorem. 
\end{proof}
Theorem \ref{Th_Hard_Thres_General} in fact also yields the result of Theorem \ref{Th_Principal} as a special case. Corollary \ref{Cor_Th_Principal} below gives a formal proof of this claim and thus provides a better understanding of Theorem \ref{Th_Principal}. In addition, Corollary \ref{Cor_Hard_Thres_General} below aims to characterize the class of sparse positive definite matrices which retain positive definiteness regardless of the thresholding subgraph. 
\begin{corollary}\label{Cor_Th_Principal}
Let $G$ be an undirected graph. Assume $A_G > 0$ for every $A > 0$. Then $G$ is a union of complete disconnected components. 
\end{corollary}
\begin{proof}
Apply Theorem \ref{Th_Hard_Thres_General} with $G$ equal to a complete graph and $H$ equal to the graph $G$ specified in the statement of the corollary. 
\end{proof}

\begin{corollary}\label{Cor_Hard_Thres_General}
Let $G$ be an undirected graph. Then $A_H > 0$ for every matrix $A \in \mathbb{P}_G$ and every subgraph $H$ of $G$ if and only if $G$ is a union of trees. 
\end{corollary}
\begin{proof}
($\Leftarrow$) This follows from the fact that every subgraph of a tree is a union of disconnected subtrees.\ \\ 
($\Rightarrow$) From Theorem \ref{Th_Hard_Thres_General}, every subgraph of $G$ must be a union of disconnected induced subgraphs. Therefore $G$ does not contain any cycle of length $\geq 3$ and thus $G$ is a union of trees. 
\end{proof}
Corollary \ref{Cor_Hard_Thres_General} states that if we start with a positive definite matrix $A$ that already has zeros according to a tree, then any further thresholding of $A$ will retain positive definiteness. 

The results above consider thresholding of elements regardless of their magnitude. In practical applications however, hard-thresholding is often performed on the smaller elements of the positive definite matrix in order to induce sparsity. A more natural question therefore would be to ask: if we threshold rather the small elements of a positive definite matrix, would the new thresholded matrix be positive definite ?  We show below that this is possible in general only if the original matrix has zeros according to a tree. 

First, we introduce some notation. We will say that the matrix $B$ is the \emph{hard-thresholded version of $A$ at level $\eta$} if $b_{ij} = a_{ij}$ when $|a_{ij}| > \eta$ or $i=j$, and $b_{ij}=0$ otherwise. 

\begin{theorem}\label{Th_hard_thresholding_level_ep}
Let $G$ be an undirected graph with $n \geq 3$ vertices and consider $\mathbb{P}_G$, the class of positive definite matrices with zeros according to $G$. The following are equivalent: 
\begin{enumerate}
\item There exists $\eta > 0$ such that the hard-thresholded version of $A$ at level $\eta$ is positive definite for every $A \in \mathbb{P}_G$; 
\item $G$ is a tree. 
\end{enumerate}
Moreover, if $G$ is a tree, the hard-thresholded version of $A$ at level $\eta$ is positive definite for every $\eta > 0$. 
\end{theorem}
\begin{proof}
(2) $\Rightarrow$ (1) If $G$ is a tree, we already know from Corollary \ref{Cor_Hard_Thres_General} that every $A \in \mathbb{P}_G$ will retain positive definiteness when thresholded with respect to any subgraph. Therefore, the thresholded matrix will always be positive definite regardless of the value of $\eta$. \ \\
(1) $\Rightarrow$ (2) We will prove the contrapositive form. Suppose $G$ is not a tree. Then $G$ contains a cycle of length $n \geq 3$. To prove the result, we need to show that for any level $\eta > 0$, there exists a matrix $A > 0$ with zeros according to $G$ such that the thresholded version of $A$ at level $\eta$ is not positive definite. It is therefore sufficient to provide a matrix $A \in \mathbb{P}_{C_n}$, where $C_n$ is the cycle graph with $n$ vertices, which loses positive definiteness when hard-thresholded at any level $\eta > 0$. 

Let $M_n$ be the matrix considered in Theorem \ref{Th_Hard_Thres_General}
\[
M_n = \left(\begin{array}{cccccccc}
2 & 1 & 0 & 0 & \dots & 0 & 0 & \mc{a}\\
1 & 2 & 1 & 0 & \dots & 0 & 0 & \mc{0}\\
0 & 1 & 2 & 1 & \dots & 0 & 0 & \mc{0}\\
\vdots & \vdots & \vdots & \vdots & \ddots & 2 & 1 & \mc{0}  \\
0 & 0 & 0 & 0 & 0 & 1 & 2 & \mc{b} \\
\cline{1-8}
a & 0 & 0 & 0 & \dots & 0 & b & \mc{2}
\end{array}\right).
\]
We know from the proof of Theorem \ref{Th_Hard_Thres_General} that we can choose $a$ and $b$ such that $M_n > 0$, and $M_n$ loses positive definiteness when $a$ is set to zero. We will now prove that in addition, we can always choose $b > 1$ and $|a| < 1$ such that $M_n > 0$ and loses positive definiteness when $a$ is set to $0$. This will allow us to obtain the thresholded matrix by thresholding at level $\epsilon_0 $ for some $|a| \leq \epsilon_0 < 1$. The general result will follow by rescaling the matrix. 

Following Theorem \ref{Th_Hard_Thres_General}, we know that if we want $M_n$ to lose positive definiteness when $a$ is set to zero, we must choose 
\begin{equation}\label{eq:b_square}
b^2 = \frac{2n}{n-1} + \epsilon
\end{equation}
 for some $\epsilon > 0$ small enough. Therefore, we can choose $b > 1$. Also, in order for $M_n$ to be positive definite, we must choose $a$ between the two roots of the polynomial $p_b(a) := p(a,b)$ when $b^2$ takes the value prescribed in (\ref{eq:b_square}) (see the proof of Theorem \ref{Th_Hard_Thres_General}).
But notice that the maximum of $p_b$ is obtained when $a = a^*$ where 
\[
a^* = (-1)^{n+1}\frac{b}{n-1}. 
\] 
Therefore, if $\epsilon$ is small enough, we have $|a^*| < 1$ and $p_b(a^*) > 0$. With those choices of $a = a^*$ and $b$, the matrix $M_n$ is positive definite and loses positive definiteness if hard-thresholded at level $|a|$. The general case, for a given level $\eta > 0$, is obtained by considering the matrix $(\eta/|a|) M_n$.  
\end{proof}

{\it Remark: } Theorem \ref{Th_hard_thresholding_level_ep} proves that if the initial matrix $A > 0$ has zeros according to a tree, and as the thresholding parameter $\eta$ is continuously increased from $0$ to $1$, the (piecewise continuous) path of thresholded matrices remain within the cone $\mathbb{P}^+$. 

Following Theorem \ref{Th_Principal}, it is natural to seek a characterization of the maximal set of matrices which will retain positive definiteness after thresholding with respect to a graph $G$. Let $\mathcal{M}(G)$ denote this set: 
\[
\mathcal{M}(G) = \{M \in \mathbb{P}_n^+ : M_G \in \mathbb{P}_n^+\}. 
\]
According to Theorem \ref{Th_Principal}, $\mathcal{M}(G) = \mathbb{P}_n^+$, i.e., every positive definite matrix retains positive definiteness when thresholded, if and only if $G$ can be written as a union of complete disconnected components. But as discussed in Section \ref{Subsec_Linalg}, $\mathcal{M}(G)$ always contains the set of diagonally dominant matrices, irrespective of the graph $G$. The following result shows that the set of diagonally dominant matrices is never maximal in the sense that, given any graph $G$, we can always find non diagonally dominant positive definite matrices which retain positive definiteness when thresholded with respect to $G$. We first prove a simple lemma required in the subsequent result. 

\begin{lemma}\label{lemma_cutpoints}Let $G = (V,E)$ be a connected graph. Then there exists a vertex $v \in V$ such that the subgraph of $G$ induced by $V \backslash \{v\}$ is connected. 
\end{lemma}
\begin{proof}
Let $u \in V$ be any vertex of $G$ and let $v \in V$ be a vertex at the maximum distance possible from $u$, i.e., every path from $u$ to a vertex $w \in V \backslash \{u,v\}$ contains less than or the same number of edges as the shortest path connecting $u$ to $v$. We claim that the subgraph $G^*$ induced by $V \backslash \{v\}$ is connected. Suppose to the contrary that this is not the case. Let $w \in V$ be a vertex that is not connected to $u$ in $G^*$. This means that every path from $u$ to $w$ in $G$ passes through $v$. As a consequence, the distance between $u$ and $w$ in $G$ is strictly greater that the distance  between $u$ and $v$. This contradicts the maximality of the distance between $u$ and $v$. The graph $G^*$ must therefore be connected. 
\end{proof}

\begin{proposition}\label{prop_dd}
Let $G = (V,E)$ be any undirected, connected graph with at least 3 vertices. Then there exists a matrix $A = (a_{ij})$ with the following properties: 
\begin{enumerate}
\item $A$ is positive definite; 
\item $A$ has no zeros;
\item For every $i$, 
\[
|a_{ii}| < \sum_{j \not= i} |a_{ij}|,  
\]  
i.e. $A$ is not diagonally dominant within any row;
\item $A_G$ is positive definite; 
\item $A_G$ is not diagonally-dominant. 
\end{enumerate}
\end{proposition}
\begin{proof}
We will prove the result by induction on $|V|$. Assume first that $|V|=3$. It is easily verified that for any graph $G$, the following matrix satisfies the theorem: 
\[
A = \left(\begin{array}{ccc}
3 & -2 & -2 \\
-2 & 3 & 2 \\
-2 & 2 & 3
\end{array}\right). 
\]
Now assume by the induction hypothesis that the theorem is true for $|V|=n-1$. Let $G$ be a connected graph with $n$ vertices. By Lemma \ref{lemma_cutpoints}, there is a vertex $v \in V$ such that the subgraph $G^*$ of $G$ induced by $V \backslash \{v\}$ is connected. By the induction hypothesis applied to $G^*$, there exists a matrix $A_{n-1} \in M_{n-1}(\mathbb{R})$ satisfying properties (1) to (5) with respect to $G^*$. Now consider the matrix  
\[
B = \left(\begin{array}{cccc}
 & & & \mc{x_1} \\
 & \lambda A_{n-1} & & \mc{x_2} \\
 & & & \mc{\vdots} \\
\cline{1-3}
x_1 & x_2 & \hdots  & x_n
\end{array}\right).
\]
We shall demonstrate that for appropriate values of $x_1, \dots, x_n$ and $\lambda$, $B$ will satisfy all the conditions of the theorem.

One can first easily  choose $x_1, \dots, x_n$ such that $B$ satisfies (2), (3) and (5). Indeed, let $x_1, \dots, x_n$ be chosen such that 
\begin{enumerate}
\item $x_i > 0$ for every $i=1, \dots, n$; 
\item $x_1 > x_n$; 
\end{enumerate} 
Condition (2) is verified since $x_i > 0$ for every $i$. Now, by the induction hypothesis, the matrix $B$ is not diagonally dominant within the first $n-1$ rows. Also, since $|x_1| > |x_n|$, the matrix is not diagonally within the last row and so condition (3) holds. Finally, $B_G$ is not diagonally dominant since, by the induction hypothesis, $(A_{n-1})_G$ is not diagonally dominant. Therefore conditions (2), (3) and (5) hold for any such choice of $x_1, \dots, x_n$ and for any $\lambda > 0$. 

We now proceed to show that we can adjust $\lambda$ in order to get (1) and (4) too.  Notice that, for $B$ and $B_G$ to be positive definite, we only need to ensure that $\det B > 0$ and $\det B_G > 0$. Let $x = (x_1, \dots x_{n-1})^t$ be the vector containing the first $n-1$ elements of the last column of $B$. Then 
\begin{eqnarray*}
\det B &=& \det (\lambda A_{n-1}) \left(x_n- \frac{1}{\lambda} x^t A_{n-1}^{-1} x\right), 
\end{eqnarray*}
which can be made to be positive for $\lambda \geq \lambda_1 > 0$, say. The same is true for $B_G$, but with $A_{n-1}$ and $x$ thresholded. Therefore, for some $\lambda_2 > 0$, $\det B_G > 0$ for $\lambda \geq \lambda_2$. The result follows by taking $\lambda = \max(\lambda_1, \lambda_2)$ in the construction of $B$. 
\end{proof}

\section{Thresholding by a chordal/decomposable graph}\label{Section_Decomp}

In this section, we will characterize the set $\mathcal{M}(G)$ for certain classes of graphs.  Recall that the proofs in the last sections hinges on constructing cycles of length $n \geq 3$. Hence a natural step in characterizing matrices which retain positive definiteness when thresholded with respect to a graph $G$ is to consider graphs without induced cycles of length $n \geq 3$ or, more generally, graphs without induced cycles of length $n \geq 4$. Those graphs correspond to trees and chordal/decomposable graphs respectively. Characterizing $\mathcal{M}(G)$ when $G$ is chordal/decomposable is the topic of study in this section. We will see that, even in simple cases, a complete characterization may involve complex algebraic relations thus giving little hope of obtaining a general simple characterization of $\mathcal{M}(G)$. 

\begin{proposition}\label{prop_decomp}
Let $(A,B,C)$ be a decomposition of a graph $G = (V,E)$, i.e., 
\begin{enumerate}
\item $V = A \cup B \cup C$; 
\item $C$ separates $A$ from $B$; 
\item $C$ is complete.
\end{enumerate}
Let $N$ be a positive definite matrix . Then $M = N_G$ is positive definite iff 
\begin{enumerate}
\item $
M_{A \cup C}  = \left(\begin{array}{cc} M_{AA} & M_{AC} \\ M_{CA} & M_{CC}\end{array}\right) > 0;
$
\item $
M_{B \cup C} = \left(\begin{array}{cc}M_{CC} & M_{CB} \\ M_{BC} & M_{BB}\end{array}\right) > 0;
$
\item $S_1 + S_2 - M_{CC} > 0$, where $S_1 = M_{CC} - M_{CA}M_{AA}^{-1} M_{AC}$ and  $S_2 = M_{CC} - M_{CB}M_{BB}^{-1}M_{BC}$ are Schur complements of $M_{AA}$ in $M_{A \cup C}$ and $M_{BB}$ in $M_{B \cup C}$ respectively. 
\end{enumerate}
Moreover, conditions (1) and (2) can be replaced by $(1')$ and $(2')$ where 
\begin{description}
 \item[$1'.$] $M_{AA} > 0$ 
\item[$2'.$] $M_{BB} > 0$. 
\end{description}
\end{proposition}
\begin{proof}
 Let us factor $M$ according to the decomposition: 
\[
M = \left(
\begin{array}{ccc}
M_{AA} & M_{AC} & 0 \\
M_{CA} & M_{CC} & M_{CB} \\
0 & M_{BC} & M_{BB}
\end{array}
\right)
\]
We know that $M$ is positive definite iff 
\[
M_{BB} > 0
\]
and 
\[
M_{A \cup C} - \left(\begin{array}{c}0 \\ M_{CB}\end{array}\right) M_{BB}^{-1} \left(\begin{array}{cc} 0 & M_{BC}\end{array}\right)  > 0.
\]
Let us develop the last term. 
\begin{eqnarray*}
M_{A \cup C} - \left(\begin{array}{c}0 \\ M_{CB}\end{array}\right) M_{BB}^{-1} \left(\begin{array}{cc} 0 & M_{BC}\end{array}\right) &=&
\left(\begin{array}{cc}
M_{AA} & M_{AC}  \\
M_{CA} & M_{CC}
\end{array}\right) - \left(\begin{array}{cc}
0 & 0 \\
0 & M_{CB}M_{BB}^{-1}M_{BC}
\end{array}\right) \\
&=& \left(\begin{array}{cc}
M_{AA} & M_{AC}  \\
M_{CA} & M_{CC} - M_{CB}M_{BB}^{-1}M_{BC}
\end{array}\right) 
\end{eqnarray*}
But the last matrix is positive definite iff $M_{AA} > 0$ and the Schur complement of $M_{AA}$ in this matrix given by 
\[
M_{CC} - M_{CB}M_{BB}^{-1}M_{BC} - M_{CA}M_{AA}^{-1} M_{AC} = S_1 + S_2 - M_{CC} > 0
\]
where $S_1 = M_{CC} - M_{CA}M_{AA}^{-1} M_{AC}$ and  $S_2 = M_{CC} - M_{CB}M_{BB}^{-1}M_{BC}$ are Schur complements of $M_{AA}$ in $M_{A \cup C}$ and $M_{BB}$ in $M_{B \cup C}$ respectively. In summary, $M$ is positive definite iff  $M_{AA} > 0$, $M_{BB} > 0$ and $S_1 + S_2 - M_{CC} > 0$. To conclude the proof, we only need to prove that $M_{AA} > 0$ and $M_{BB} > 0$ can be replaced by $M_{A \cup C} > 0$ and $M_{B \cup C} > 0$ respectively. 

Note that if $M > 0$, then $M_{A \cup C} > 0$ and $M_{B \cup C} > 0$ since they are principal submatrices of $M$. Therefore, $M > 0$ implies the three conditions of the theorem. Conversely, if the three conditions of the theorem are satisfied, then in particular, $M_{AA} > 0$ and $M_{BB} > 0$ as they are principal submatrices of $M_{A \cup C}$ and $M_{B \cup C}$. From the argument above, this together with the condition of the theorem $S_1 + S_2 - M_{CC} > 0$ implies that $M > 0$. 
\end{proof}
{\it Remark:} We note that another way to prove Proposition \ref{prop_decomp} is to factor $M$ using a Cholesky type decomposition: 
\[
M = \left(
\begin{array}{ccc}
M_{AA} & 0 & M_{CA}^t \\
0 & M_{BB} & M_{CB}^t \\
M_{CA} & M_{CB} & M_{CC}
\end{array}
\right) = \left(
\begin{array}{ccc}
M_{AA} & 0 & 0 \\
0 & M_{BB} & 0 \\
M_{CA} & M_{CB} & I
\end{array}
\right) \left(
\begin{array}{ccc}
M_{AA}^{-1} & 0 & 0 \\
0 & M_{BB}^{ -1} & 0 \\
0 & 0 & S
\end{array}
\right) \left(
\begin{array}{ccc}
M_{AA} & 0 & 0 \\
0 & M_{BB} & 0 \\
M_{CA} & M_{CB} & I
\end{array}
\right)^t. 
\]
When the graph $G$ is chordal/decomposable, we can apply Proposition \ref{prop_decomp} recursively to characterize the class of the positive definite matrices that retain positive definiteness when thresholded. 

\begin{theorem}\label{thm_decomp}
Let $G$ be a connected chordal/decomposable graph and let $(C_1,\dots,C_k)$ be a perfect order of its cliques. Let $H_q = C_1 \cup \dots \cup C_q$ and $S_q = C_q \cap H_{q-1}$ ($q = 2, \dots, k$). Moreover let $A_q = H_{q-1} \backslash S_q$ and $B_q = C_q \backslash S_q$. Assume $N$ is a positive definite matrix. Then $M = N_G$ is positive definite if and only if for every $2 \leq q \leq k$, 
\[
S_1^{(q)} + S_2^{(q)} - M_{S_q S_q} > 0
\]
where $S_1^{(q)} = M_{S_qS_q} - M_{S_qA_q}M_{A_q A_q}^{-1}M_{A_qS_q}$ and $S_2^{(q)} = M_{S_qS_q} - M_{S_qB_q} M_{B_qB_q}^{-1}M_{B_qS_q}$ are the Schur complements of $M_{A_q A_q}$ in $M_{A_q \cup S_q, A_q \cup S_q}$ and of $M_{B_q B_q}$ in $M_{S_q \cup B_q, S_q \cup B_q}$ respectively. 
\end{theorem}
\begin{proof}
We will prove the theorem by applying  Proposition \ref{prop_decomp} recursively to different decompositions of $G$ that correspond to the perfect ordering $(C_1, \dots, C_k)$. More specifically, we shall work backwards starting from $C_k$ then to $C_{k-1}$, etc., eventually leading up to $C_1$. 

Consider first the decomposition $(B_k, A_k, S_k)$ of $G$  (see for example \cite{lauritzen} Lemma 2.11). By an application of Proposition  \ref{prop_decomp}, $M > 0$ if and only if 
\begin{enumerate}
\item $M_{C_k C_k} > 0$; 
\item $M_{H_{k-1} H_{k-1}} > 0$; 
\item $S_1^{(k)} + S_2^{(k)} - M_{S_k S_k} > 0$. 
\end{enumerate}
Notice that condition (1) is trivially verified since $N > 0$ and $C_k$ is complete. Consider now the subgraph $G_{k-1}$ induced by $H_{k-1}$. A decomposition of this subgraph is given by $(B_{k-1}, A_{k-1}, S_{k-1})$. By the same argument as above, the condition $M_{H_{k-1}, H_{k-1}} > 0$ is equivalent to 
\begin{enumerate}
\item $M_{C_{k-1} C_{k-1}} > 0$; 
\item $M_{H_{k-2} H_{k-2}} > 0$; 
\item $S_1^{(k-1)} + S_2^{(k-1)} - M_{S_{k-1} S_{k-1}} > 0$. 
\end{enumerate}
As a consequence, $M > 0$ if and only if $M_{H_{k-2} H_{k-2}} > 0$ and 
\[
S_1^{(q)} + S_2^{(q)} - M_{S_q S_q} > 0 \qquad (q = k-1, k). 
\]
By applying the same reasoning to the graphs $G_q$ induced by $H_{q}$ $(2 \leq q \leq k-2)$ and working backwards, we obtain that $M > 0$ if and only if $M_{H_1 H_1} > 0$ and 
\[
S_1^{(q)} + S_2^{(q)} - M_{S_q S_q} > 0 \qquad (2 \leq q \leq k).
\]
Since $H_1 = C_1$, the condition  $M_{H_1 H_1} > 0$ is trivially satisfied and so $M > 0$ if and only if 
\[
S_1^{(q)} + S_2^{(q)} - M_{S_q S_q} > 0 \qquad (2 \leq q \leq k). 
\]
\end{proof}

The next corollary shows that the conditions of the preceding theorem reduce to scalar conditions when the graph is a tree. 
\begin{corollary}\label{cor_tree_charac}
Let $T$ be a tree with root $v_0$. Let $e_i = (p_i,q_i)$, $i=1, \dots, k$, be the edges of $T$ labeled according to Figure \ref{fig:tree}, i.e., starting from left and proceeding to the right at every depth of the tree. Note that $p_i$ is the parent of $q_i$. Let $N > 0$ and define $M = N_G$. 
\begin{figure}[h]
\begin{center}
\includegraphics[width=8cm]{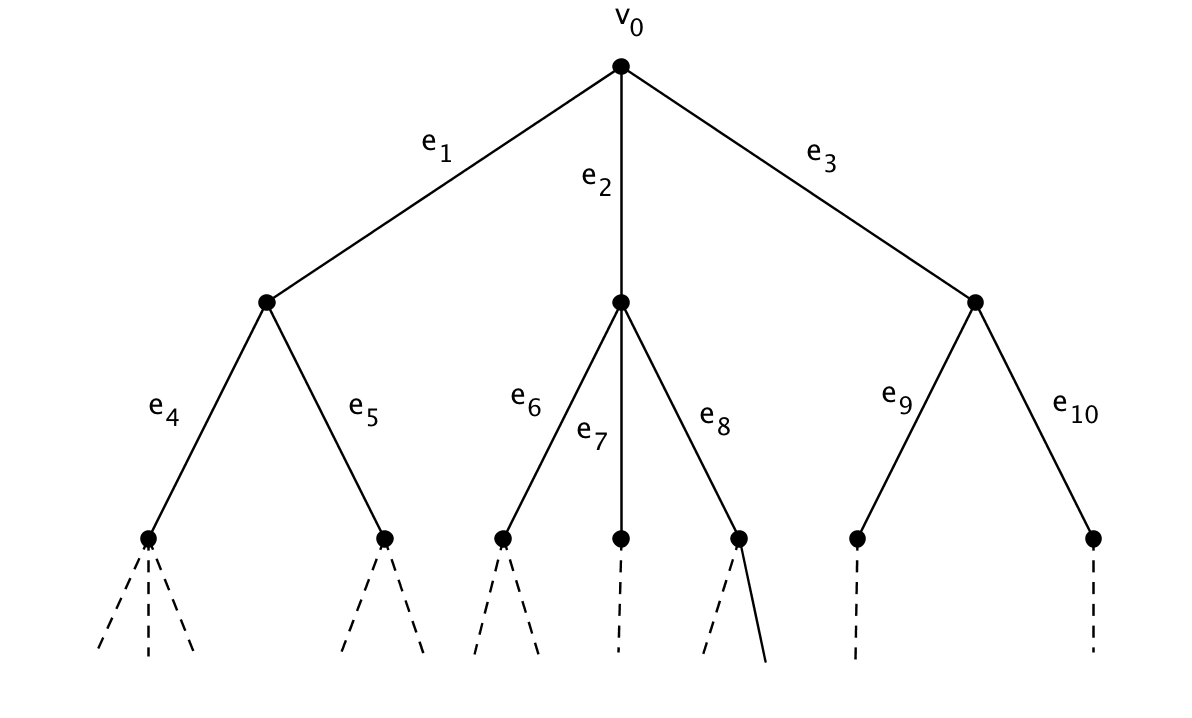}
\end{center}
\caption{Edge enumeration of a tree}
\label{fig:tree}
\end{figure}
For each $j=2, \dots, k$, define a scalar $\sigma_j$ as follows: 
\[
\sigma_j = M_{p_j,p_j} - M_{p_j,A_j} M_{A_j,A_j}^{-1} M_{A_j,p_j}
\]
where $A_j = \cup_{i=1}^{j-1} \{p_i, q_i\} \backslash \{p_j\}$. Also, let 
\[
\eta_j = M_{p_j,p_j} - M_{p_j,q_j}^2 M_{q_j,q_j}^{-1}. 
\]
Then $M > 0$ if and only if 
\[
\sigma_j + \eta_j - M_{p_j,p_j} > 0
\]
for every $j=2, \dots, k$. 
\end{corollary}
\begin{proof}
This follows from Theorem \ref{thm_decomp} by noting that the enumeration of the edges given above is in fact a perfect ordering of the cliques of $T$. 
\end{proof}
Theorem \ref{thm_decomp} and Corollary \ref{cor_tree_charac} demonstrate that for decomposable graphs, or even simpler graphs like trees, characterizing the class of matrices retaining positive definiteness can lead to complex algebraic conditions. For a narrow class of graphs however, we can give a more explicit characterization of the matrices retaining positive definiteness. This is the case when the graph is ``auto-similar'' like a path for example. We first prove a lemma that is a key ingredient for our next result. 

\begin{lemma}\label{lemma_schur_path}
Let $A_n$ be a general matrix over a path of length $n \geq 3$, given as follows 
\[
A_n = \left(\begin{array}{ccccccc}
\alpha_1 & \mc{a_1} & & & & &\\
\cline{1-6}
a_1 & \mc{\alpha_2} & a_2 & & & & \\
 & \mc{a_2} & \alpha_3 & a_3 & & &\\
& \mc{}& a_3 & \alpha_4 & a_4 & &\\
 & \mc{ }& & & \ddots & &\\
 & \mc{}& & a_{n-1}& \alpha_{n-1} & a_{n-1} \\
 & \mc{}& & & a_{n-1} & \alpha_n
\end{array}\right) 
\]
where the blank entries are zeros. Then the Schur complement of the lower right block in $A_n$, denoted $\sigma_n(1)$, is given by the following continued fraction: 
\[
\sigma_n(1) := \alpha_1 - \cfrac{a_1^2}{\alpha_2 - \cfrac{a_2^2}{\alpha_3 - \cfrac{a_3^2}{\alpha_4 - \cfrac{\dots}{\alpha_{n-1}- \cfrac{a_{n-1}^2}{\alpha_n}}}}}.
\]
\end{lemma}
\begin{proof}
We shall prove the result by induction. Assume first that $n=3$. Then 
\[
A_3 = \left(\begin{array}{ccc}
\alpha_1 & a_1 & 0 \\
a_1 & \alpha_2 & a_2 \\
0 & a_2 & \alpha_3
\end{array}\right)
\]
and the Schur complement of the lower right block in $A_3$ is given by 
\begin{eqnarray*}
\sigma_3(1) &=& \alpha_1 - (\begin{array}{cc}a_1 & 0\end{array}) \left(\begin{array}{cc}\alpha_2 & a_2 \\ a_2 & \alpha_3\end{array}\right)^{-1} \left(\begin{array}{c}a_1 \\ 0\end{array}\right) \\
&=& \alpha_1 - (\begin{array}{cc}a_1 & 0\end{array}) \frac{1}{\alpha_2\alpha_3 - a_2^2}\left(\begin{array}{cc}\alpha_3 & -a_2 \\ -a_2 & \alpha_2\end{array}\right) \left(\begin{array}{c}a_1 \\ 0\end{array}\right) \\
&=& \alpha_1 - \frac{a_1^2\alpha_3}{\alpha_2\alpha_3 - a_2^2} \\
&=& \alpha_1 - \cfrac{a_1^2}{\alpha_2 - \cfrac{a_2^2}{\alpha_3}}. 
\end{eqnarray*}
Now assume the expression is true for the $(n-1) \times (n-1)$ case and consider the $n \times n$ matrix as given in the statement of the lemma. The Schur complement of the lower right block in $A_n$ is given by 
\begin{eqnarray*}
\sigma_n(1) &=& \alpha_1 - (\begin{array}{cccc}a_1 & 0 & \dots & 0\end{array})\left(\begin{array}{ccccc}\alpha_2 & a_2 & & & \\
a_2 & \alpha_3 & a_3 & & \\
 & a_3 & \alpha_4 & a_4 & \\
 & & & \ddots & \\
 & & a_{n-1} & \alpha_{n-1} & a_n \\
 & & & a_{n-1} & \alpha_n\end{array}\right)^{-1}\left(\begin{array}{c}a_1 \\ 0 \\ \vdots \\ 0\end{array}\right) \\
&=& \alpha_1 - (\begin{array}{cccc}a_1 & 0 & \dots & 0\end{array}) \tilde{A}^{-1}\left(\begin{array}{c}a_1 \\ 0 \\ \vdots \\ 0\end{array}\right)
\end{eqnarray*}
where we have labeled $\tilde{A}$ the lower right block of $A_n$. Now, notice that the righthand term of the last expression above is equal to $a_1^2$ times the $(1,1)$ element of $\tilde{A}^{-1}$. Recall that the upper left block of the inverse of a block matrix $M$ denoted by 
\[
M = \left(\begin{array}{cc}A & B \\ C & D\end{array}\right)
\]
is given by the inverse of the Schur complement of $D$ in $M$. Therefore, if $S$ is the Schur complement of $\tilde{A}$ in $A_n$, then  
\[
\sigma_n(1) = \alpha_1 - a_1^2 S^{-1}. 
\]
But by the induction hypothesis, 
\[
S = \alpha_2 - \cfrac{a_2^2}{\alpha_3 - \cfrac{a_3^2}{\alpha_4 - \cfrac{\dots}{\alpha_{n-1}- \cfrac{a_n^2}{\alpha_n}}}}
\]
and so the result follows. 
\end{proof}

We now give algebraic conditions that characterize the class of p.d. matrices which retain positive definiteness when thresholded with respect to a path. We first note that the result below can be derived from first principles by using Lemma \ref{lemma_schur_path} and properties of positive definite matrices. Since a path is also a decomposable graph, we show that Proposition \ref{prop_decomp} can be useful in discovering such characterizations. We illustrate this idea in the next corollary.

\begin{corollary}\label{cor_charac_path}
Let $G$ be a path of length $n \geq 3$ and let $N$ be a $n \times n$ positive definite matrix. Let $M = N_G$ given as in Lemma \ref{lemma_schur_path}. Then the thresholded matrix $M$ is positive definite iff: 
\begin{equation}\label{eqn_path_sigma1}
\sigma_n(k+1) > a_k^2 \alpha_k^{-1} \qquad \forall k=1,\dots,n-2, 
\end{equation}
or equivalently, iff
\begin{equation}\label{eqn_path_sigma2}
\sigma_n(k) > 0 \qquad \forall k=1,\dots,n-2 
\end{equation}
where 
\[
\sigma_n(k) := \alpha_k - \cfrac{a_k^2}{\alpha_{k+1} - \cfrac{a_{k+1}^2}{\alpha_{k+2}-\cfrac{\dots}{\alpha_{n-1}-\cfrac{a_{n-1}^2}{\alpha_n}}}}.
\]
\end{corollary}
\begin{proof}
We prove the result by induction on $n$. Suppose first that $n=3$. Note that in this case, we only need to check (\ref{eqn_path_sigma1}) for $k=1$. Since $N > 0$, the matrix $M = N_G$ is positive definite if and only if $\det M > 0$. Now 
\begin{equation}\label{eqn_det1}
\det M = \alpha_1(\alpha_2 \alpha_3 - a_2^2) - a_1^2 \alpha_3. 
\end{equation}
On the other hand, equation (\ref{eqn_path_sigma1}) can be expressed as 
\begin{equation}\label{eqn_det2}
 \sigma_3(2) - a_1^2\alpha_1^{-1} =   \alpha_2 - \frac{a_2^2}{\alpha_3}  - a_1^2\alpha_1^{-1} = \frac{1}{\alpha_1 \alpha_3} \det M. 
\end{equation}
Since $N > 0$, the diagonal elements of $M$ are positive and therefore the two conditions (\ref{eqn_det1}) and (\ref{eqn_det2}) are equivalent. This proves the result for $n=3$. Now assume the result is true for every path with $n-1$ vertices and let $G$ be a path with $n$ vertices. Let $A = \{1\}$, $C = \{2\}$ and $B = \{3,\dots,n\}$. Then $(A, C, B)$ is a decomposition of the path $G$ and by Proposition \ref{prop_decomp}, $M > 0$ iff  $M_{A \cup C} > 0$, $M_{B \cup C} > 0$ and $S_1 + S_2 - M_{CC} > 0$ where $S_1$ and $S_2$ are the Schur complements of $M_{AA}$ and $M_{BB}$ in $M_{A \cup C}$ and $M_{B \cup C}$ respectively. We have $S_1 = \alpha_2 - a_1^2 \alpha_1^{-1}$ and by Lemma \ref{lemma_schur_path}, $S_2 = \sigma_n(2)$. Notice that $M_{A \cup C} > 0$ since $M > 0$ and $A \cup C$ is complete as it is not affected by the thresholding. Therefore, $M > 0$ iff $M_{B \cup C} > 0$ and $S_1 + S_2 - M_{CC} = \alpha_2 - a_1^2 \alpha_1^{-1} + \sigma_n(2) - \alpha_2 = \sigma_n(2)-a_1^2 \alpha_1^{-1} > 0$. The latter condition is equivalent to equation (\ref{eqn_path_sigma1}) for $k=1$. Now, $M_{B \cup C}$ is the same matrix as $M$ but for a path of $n-1$ points. Therefore, by the induction hypothesis, the matrix $M_{B \cup C}$ is positive definite iff 
\[
\sigma_n(k+1) - a_k^2 \alpha_{k}^{-1} > 0 \qquad \forall k=2,\dots,n-2. 
\]
As a consequence, the matrix $M = N_G$ is positive definite if and only if 
\[
\sigma_n(k+1) > a_k^2 \alpha_k^{-1} \qquad \forall k=1,\dots,n-2
\]
and the result follows. The preceding condition can be simplified to $\sigma_n(k) > 0$ for $k=1, \dots, n-2$ since 
\[
\sigma_n(k) = \alpha_k - \frac{a_k^2}{\sigma_n(k+1)}. 
\]
This proves the equivalence between conditions (\ref{eqn_path_sigma1}) and (\ref{eqn_path_sigma2}) and concludes the proof. 
\end{proof}

\begin{acknowledgement}
We wish to thank Professor Stanley C. Eisenstat for useful feedback on the paper. 
\end{acknowledgement}


\bibliography{biblio}

\begin{thebibliography}{1}

\bibitem{bickel_levina}
Peter~J. Bickel and Elizaveta Levina.
\newblock Covariance regularization by thresholding.
\newblock {\em Ann. Statist.}, 36(6):2577--2604, 2008.

\bibitem{Chavatal}
V.~Chv{\'a}tal.
\newblock Remark on a paper of {L}ov\'asz.
\newblock {\em Comment. Math. Univ. Carolinae}, 9:47--50, 1968.

\bibitem{Golumbic}
Martin~Charles Golumbic.
\newblock {\em Algorithmic graph theory and perfect graphs}, volume~57 of {\em
  Annals of Discrete Mathematics}.
\newblock Elsevier Science B.V., Amsterdam, second edition, 2004.
\newblock With a foreword by Claude Berge.

\bibitem{HJ_Matrix}
Roger~A. Horn and Charles~R. Johnson.
\newblock {\em Matrix analysis}.
\newblock Cambridge University Press, Cambridge, 1990.
\newblock Corrected reprint of the 1985 original.

\bibitem{lauritzen}
Steffen~L. Lauritzen.
\newblock {\em Graphical models}, volume~17 of {\em Oxford Statistical Science
  Series}.
\newblock The Clarendon Press Oxford University Press, New York, 1996.
\newblock Oxford Science Publications.

\bibitem{Li_Horvath}
Ai~Li and Steve Horvath.
\newblock {Network neighborhood analysis with the multi-node topological
  overlap measure}.
\newblock {\em Bioinformatics}, 23(2):222--231, 2007.

\bibitem{flexible_cov}
Bala Rajaratnam, H{\'e}l{\`e}ne Massam, and Carlos~M. Carvalho.
\newblock Flexible covariance estimation in graphical {G}aussian models.
\newblock {\em Ann. Statist.}, 36(6):2818--2849, 2008.

\bibitem{Zhang_Horvath}
Bin Zhang and Steve Horvath.
\newblock A general framework for weighted gene co-expression network analysis.
\newblock {\em Stat. Appl. Genet. Mol. Biol.}, 4:Art. 17, 45 pp. (electronic),
  2005.

\end{thebibliography}

\end{document}